\documentclass{llncs}
\usepackage{epsfig}
\usepackage{amsmath}
\usepackage{amsfonts}
\usepackage{amssymb}

\newcommand{\R}{{I\!\!R}}

\def\R{{\rm I}\! {\rm R}}

\def\ss{sequential splitting}

\newtheorem{algorithm}[theorem]{Algorithm}
\newtheorem{assumption}[theorem]{Assumption}
\begin{document}

\pagestyle{headings}

\title{Iterative and Non-iterative Splitting approach of a stochastic Burgers' equation}
\author{J\"urgen Geiser and Karsten Bartecki}
\institute{Ruhr University of Bochum, \\
The Institute of Theoretical Electrical Engineering, \\
Universit\"atsstrasse 150, D-44801 Bochum, Germany \\
\email{juergen.geiser@ruhr-uni-bochum.de}}
\maketitle

\begin{abstract}

In this paper we present iterative and noniterative splitting methods,
which are used to solve stochastic Burgers' equations.
The non-iterative splitting methods are based on Lie-Trotter and
Strang-splitting methods, while the iterative splitting
approaches are based on the exponential integrators for
stochastic differential equations. Based on the nonlinearity of the
Burgers' equation, we have investigated that the iterative schemes
are more accurate and efficient as the non-iterative methods.

\end{abstract}



\section{Introduction}

We are motivated to model the nonlinear transport phenomenon
with stochastic perturbations.
Such modelling problems arise in many fields, such as biology, physics, engineering and economics, where random phenomena play an important role. 
We concentrate on nonlinear transport with stochastic reaction, which 
can be modelled by the Burgers' equation with an additional stochastic part.
Many applications in transport phenomena can be modelled
with uncertainties in combining deterministic and stochastic operators, see \cite{karlson2017}.
To solve such delicate problems, we consider operator splitting approaches to decompose into deterministic and stochastic operators, see \cite{geiser2019_1}
and \cite{ninomya2008}.
The numerical schemes are discussed as noniterative schemes in the
direction of Lie-Trotter-splitting and Strang-splitting schemes, see \cite{trotter59} and \cite{stra68},
which we call AB-, ABA-, BAB-schemes. The iterative schemes are discussed
in the direction of Picard's iterative schemes, see \cite{geiser_2011}.
We apply the extension of the deterministic to the stochastic schemes,
which are given in the numerical stochastics literature, e.g., \cite{kloeden1992}, 
\cite{oksendal02} and \cite{evans2013}.

The benefit of splitting approaches arises in decomposing different 
operators, which can be solved numerically with more optimal methods.
In the underlying stochastic Burgers' equation, we decompose the deterministic
part, which has to be solved with fast conservation methods, see \cite{karlson2018} and \cite{karlson2017},
and the stochastic part, which has to be solved as a stochstatic ordinary
differential equation, see \cite{karlson2018} and \cite{karlson2017}.

The paper is outlined as following:
The model is given in section \ref{modell}.
The numerical methods are discussed in section \ref{num}.
The numerical analysis is presented in section \ref{analysis}.
In section \ref{numerics}, we present the numerical simulations.
In the contents, that are given in section \ref{concl}, 
we summarize our results.

\section{Mathematical Model}
\label{modell}

In modeling, we concentrate on nonlinear stochastic PDEs (SPDEs),
which are important to fluid dynamics. Here, we deal with
stochastic Burgers equation (SBE) driven by linear multiplicative noise,
see \cite{hairer2010} and \cite{iliescu2018}.

The SBE is given as:
\begin{eqnarray}
\label{stoch_1}
&& d c = (\nu u_{xx} - f(c)_x) dt +  \sigma(c) \; d W_t ,  \; (x, t) \in [0, X] \times [0, T] , \\
&& c(x, 0) = c_0(x) , \; x \in [0, X] , \\
&& c(0, t) = c(X, t) = 0 , \; \in t \ge 0 ,
\end{eqnarray}
where $\nu$ is a positive diffusion coefficient, $W_t$ is a two-sided one-dimensional Wiener process. $f(c)$ is the nonlinear flux-function, e.g., $f(c) = c^2$. Further $f$ is globally Lipschitz continuous in $\R$. $\sigma$ is multiplicative noise function and Lipschitz continuous in $c$, which measures the amplitude of the noise. $c_0$ is an initial condition. 

Such SPDEs driven by linear multiplicative noise and especially the SBE (\ref{stoch_1}) are used to model turbulences or non-equilibrium phase transitions, see
\cite{birnir2013} and \cite{munoz2004}. Further, the models are used to deal with randomly fluctuating environments \cite{bedd1977} and also to model of parameter 
disturbances based on uncertainties, see \cite{bloemker2007}.

We deal with a stochastic balance equation, which is given in the nonlinear
transport case as a pure stochastic Burger's equation:
\begin{eqnarray}
\label{stoch_1}
&& d c +  (\frac{\partial f(c)}{\partial x}) dt = \sigma(c) d W_t ,  \; (x, t) \in [0, X] \times [0, T] , \\
&& c(x, 0) = c_0(x) , \; x \in [0, X] ,  \\
&& c(0, t) = c(X, t) = 0 , \; \in t \ge 0 ,
\end{eqnarray}
where $\sigma$ is the multiplicative noise function, $f(c)$ is the nonlinear flux function and $W(t)$ is a Wiener process.

In the next section, we deal with the different solver methods.

\section{Numerical Methods}
\label{num}

For the numerical methods, we apply based on the idea
of separating the deterministic and stochastic operators, the following 
numerical approaches, see also \cite{karlson2017,karlson2018}:
\begin{itemize}
\item Deterministic solver for the pure Burgers' equation: Finite volume discretization for the space with the conservation law solver of Engquist-Osher, see \cite{harten1986} and \cite{holden2015}.
\item Stochastic solver for the pure stochastic ODE: Euler-Maruyama solver, Milstein solver or stochastic RK solver, see \cite{kloeden1992} and \cite{roessler2006}.
\end{itemize}

The application of the separated solver methods is done with different 
splitting approaches, see an overview in \cite{geiser_2011}.

We have the following parts of the full equation (\ref{stoch_1}).
\begin{eqnarray}
&& \frac{\partial c}{\partial t} +  \frac{\partial f(c)}{\partial x} = \sigma(c) \frac{\partial W}{\partial t} ,  
\end{eqnarray}
where 

\begin{itemize}
\item The deterministic part:
\begin{eqnarray}
\label{determ_1}
&& \frac{\partial c}{\partial t} +  \frac{\partial f(c)}{\partial x} = 0 , \; c(x, 0) = c_0(x) ,  
\end{eqnarray}
where we have the solution
\begin{eqnarray}
&& c(t) = S_{Burgers} c_0 ,  
\end{eqnarray}
\item The stochastic part:
\begin{eqnarray}
\label{stoch_1}
&& d c = \sigma(c) \; d W , \; c(x, 0) = c_0(x) ,  
\end{eqnarray}
where we have the solution
\begin{eqnarray}
&& c(x,t) = c_0(x) + \int_{0}^t \sigma(c) \; d W .  
\end{eqnarray}

\end{itemize}

We concentrate on the following methods:

\begin{itemize}
\item Noniterative methods based on exponential splitting approaches:
\begin{itemize}
\item AB-splitting (Lie-Trotter scheme, see \cite{karlson2017}),
\item ABA splitting (Strang-splitting scheme, see \cite{ninomya2008}),
\item BAB splitting (Strang-splitting scheme, see \cite{ninomya2008}),
\end{itemize}
\item Iterative method based on successive relaxation approaches:
\begin{itemize}
\item Iterative splitting (Picard's approximation, see \cite{geiser2019_1}).
\end{itemize}
\end{itemize}

In the following, we discuss the different schemes.

\subsection{Noniterative splitting approaches}

The noniterative splitting approaches obtained results in one cycle,
which means it is not necessary to relax the solution.
We consider the ideas related to the exponential splitting based on the
Lie-Trotter schemes, see \cite{trotter59} and \cite{geiser_2011}, while we compute the numerical results 
for each operator-equation, see equation (\ref{determ_1})-(\ref{stoch_1}) and couple the results as an initial value of the successor operator-equation, for example, we apply the
results of equation (\ref{determ_1}) as an initial value for the equation (\ref{stoch_1}), see also 
\cite{geiser_2011}.

\begin{enumerate}
\item AB splitting: \\

We have the following AB splitting approaches:

\begin{itemize}
\item A-Part
\begin{eqnarray}
&& \frac{\partial c_1}{\partial t} +  \frac{\partial f(c_1)}{\partial x} = 0 , \; c_1(x, t^n) = c(x, t^n) ,  
\end{eqnarray}
where we have the solution
\begin{eqnarray}
&& c_1(t^{n+1}) = S_{Burgers} c_1(t^n) ,  
\end{eqnarray}
\item B-part:
\begin{eqnarray}
&& d c_2 = \sigma(c_2) \; d W , \; c_2(x, t^n) = c_1(x, t^{n+1}) ,  
\end{eqnarray}
where we have the solution
\begin{eqnarray}
&& c_2(x,t^{n+1}) = c_2(x, t^n) + \int_{t^n}^{t^{n+1}} \sigma(c_2) \; d W .  
\end{eqnarray}
where we have the solution $c(x, t^{n+1}) =  c_2(x,t^{n+1})$.
\end{itemize}

\item ABA splitting: \\

We have the following ABA splitting approaches:
\begin{itemize}
\item A-Part ($\Delta t/2$)
\begin{eqnarray}
&& \frac{\partial c_1}{\partial t} +  \frac{\partial f(c_1)}{\partial x} = 0 , \; c_1(x, t^n) = c(x, t^n) ,  
\end{eqnarray}
where we have the solution
\begin{eqnarray}
&& c_1(t^{n+1/2}) = S_{Burgers} c_1(t^n) , \mbox{with } \Delta t/2 . 
\end{eqnarray}
\item B-part:
\begin{eqnarray}
&& d c_2 = \sigma(c_2) \; d W , \; c_2(x, t^n) = c_1(x, t^{n+1/2}) ,  
\end{eqnarray}
where we have the solution
\begin{eqnarray}
&& c_2(x,t^{n+1}) = c_2(x, t^{n+1/2}) + \int_{t^n}^{t^{n+1}} \sigma(c_2) \; d W .  
\end{eqnarray}
\item A-Part ($\Delta t/2$)
\begin{eqnarray}
&& \frac{\partial c_3}{\partial t} +  \frac{\partial f(c_3)}{\partial x} = 0 , \; c_3(x, t^{n+1/2}) = c_2(x, t^{n+1}) ,  
\end{eqnarray}
where we have the solution
\begin{eqnarray}
&& c_3(t^{n+1/2}) = S_{Burgers} c_2(t^{n+1}) , \mbox{with } \Delta t/2 . 
\end{eqnarray}
where  we have the solution $c(x, t^{n+1}) =  c_3(x,t^{n+1})$.

\end{itemize}

\item BAB splitting: \\

We have the following BAB splitting approach:
\begin{itemize}
\item B-part ($\Delta t/2$):
\begin{eqnarray}
&& d c_1 = \sigma(c_1) \; d W , \; c_2(x, t^n) = c(x, t^{n}) ,  
\end{eqnarray}
where we have the solution
\begin{eqnarray}
&& c_1(x,t^{n+1/2}) = c(x, t^{n}) + \int_{t^n}^{t^{n+1/2}} \sigma(c_2) \; d W , \\
&& c_1(x,t^{n+1/2}) = S_{Stoch.} c(x, t^{n}) . 
\end{eqnarray}
\item A-Part ($\Delta t$)
\begin{eqnarray}
&& \frac{\partial c_2}{\partial t} +  \frac{\partial f(c_2)}{\partial x} = 0 , \; c_2(x, t^n) = c_1(x, t^{n+1/2}) ,  
\end{eqnarray}
where we have the solution
\begin{eqnarray}
&& c_2(t^{n+1}) = S_{Burgers} c_1(t^{n+1/2}) , \mbox{with } \Delta t . 
\end{eqnarray}
\item B-part:
\begin{eqnarray}
&& d c_3 = \sigma(c_3) \; d W , \; c_3(x, t^{n+1/2}) = c_2(x, t^{n+1}) ,  
\end{eqnarray}
where we have the solution
\begin{eqnarray}
&& c_3(x,t^{n+1}) = c_2(x, t^{n+1}) + \int_{t^{n+1/2}}^{t^{n+1}} \sigma(c_2) \; d W .  
\end{eqnarray}
where we have the solution
\begin{eqnarray}
&& c_3(t^{n+1}) = S_{Stoch.} c_2(t^{n+1}) , \mbox{with } \Delta t/2 . 
\end{eqnarray}
where  we have the solution $c(x, t^{n+1}) =  c_3(x,t^{n+1})$.

\end{itemize}

\end{enumerate}

\subsection{Iterative splitting}

The iterative splitting approaches are based on successive relaxation, means 
we apply several times in the same time-interval the solver method and improve
cyclic the solutions in this local time-interval, see \cite{geiser_2011}. 

To apply the iterative approaches, we can apply the iterative solvers
before or after a spatial discretization, means:
\begin{itemize}
\item 1.) Iterative splitting after the discretization, we apply iterative schemes for the nonlinearities.
\item 2.) Iterative splitting before the discretization, we apply the
iterative scheme to decompose the differential equation into a kernel and
perturbation term.
\end{itemize}

\subsubsection{Iterative scheme after discretization}

We have the following SDE in continuous form:
\begin{eqnarray}
&& \frac{\partial c}{\partial t} +  \frac{\partial f(c)}{\partial x} =  \sigma(c) \; \frac{d W}{dt} , \; c(x, t^n) = c(x, t^n) ,  
\end{eqnarray}
and in the SDE form as:
\begin{eqnarray}
&&  d c = A(c) dt + B(c) d W , \; c(x, t^n) = c(x, t^n) ,  
\end{eqnarray}

We apply the discretization in time (Milstein-scheme) and space (finite-volume scheme) and obtain:
\begin{eqnarray}
\label{stand_mil}
&&  c^{n+1} = c^n + A(c^n) \Delta t + B(c^n) \Delta W +  \frac{1}{2} \; B(c^n) \frac{\partial B(c)}{\partial c}|_{c = c^n} (\Delta W^2 - \Delta t) , 
\end{eqnarray}
where we have the initialization $c^n = c(x, t^n)$.

Further the solution of the Burgers' equation is given as:
\begin{eqnarray}
&&  c^{n+1} = c^n + A(c^n) \Delta t , \\
&&   c^{n+1} = S_{Burgers}(\Delta t, c^n) c^n , 
\end{eqnarray}
while we apply $c^n$ for the linearization in the Burgers' equation.

We apply a fixpoint-scheme to improve the standard Milstein scheme (\ref{stand_mil}) and obtain:
\begin{eqnarray}
\label{mod_mil}
&&  c_i^{n+1} =  S_{Burgers}(\Delta t, c_{i-1}^{n+1}) c^n + B(c_{i-1}^{n+1}) \Delta W + \nonumber \\ 
&& +  \frac{1}{2} \; B(c_{i-1}^{n+1}) \frac{\partial B(c)}{\partial c}|_{c = c_{i-1}^{n+1}} (\Delta W^2 - \Delta t) , 
\end{eqnarray}
where we have the initialization $c_0^{n+1} = c^n$.

We deal with the following iterative splitting approaches:
\begin{itemize}
\item Standard Milstein-scheme of second order ($i=1$):
\begin{eqnarray}
\label{mod_mil}
&&  c_1^{n+1} =  S_{Burgers}(\Delta t, c^{n}) c^n + B(c^{n}) \Delta W + \nonumber \\ 
&& +  \frac{1}{2} \; B(c^{n}) \frac{\partial B(c)}{\partial c}|_{c = c^{n}} (\Delta W^2 - \Delta t) , 
\end{eqnarray}
\item Second order iterative splitting approach (related to the standard Milstein-scheme of second order) ($i > 1$):
\begin{eqnarray}
&& c_i(t^{n+1}) =  S_{Burgers}(\Delta t, c_{i-1}^{n+1}) c^n + B(c_{i-1}^{n+1}) \Delta W + \nonumber \\ 
&& +  \frac{1}{2} \; B(c_{i-1}^{n+1}) \frac{\partial B(c)}{\partial c}|_{c = c_{i-1}^{n+1}} (\Delta W^2 - \Delta t) ,
\end{eqnarray}
where $\Delta W_{t^{n+1}} = W_{t^{n+1}} - W_{t^n} = \sqrt{\Delta t} \; \xi$ and $\xi$ obeys the Gaussian normal distribution $N(0, 1)$ with $\langle \xi \rangle = 0$ and $ \langle \xi^2 \rangle = 1$.

\end{itemize}

\subsubsection{Iterative scheme before the dicretization}

We have the following iterative splitting approaches, before the
discretization:

We have $i = 1, 2, \ldots, I$ with:
\begin{eqnarray}
&& \frac{\partial c_i}{\partial t} +  \frac{\partial f(c_i)}{\partial x} =  \sigma(c_{i-1}) \; \frac{d W}{dt} , \; c_i(x, t^n) = c(x, t^n) ,  
\end{eqnarray}
where we have the initialization $c_0(x, t) = c(x, t^n)$.

We have the solution
\begin{eqnarray}
&& \int_{t^n}^{t^{n+1}} \frac{\partial c_i}{\partial t} dt =  - \int_{t^n}^{t^{n+1}}  \frac{\partial f(c_i)}{\partial x} +   \int_{t^n}^{t^{n+1}} \sigma(c_{i-1}(t)) \; d W_t , \; c_i(x, t^n) = c(x, t^n) ,  \nonumber \\
&& c_i(t^{n+1}) = S_{Burgers}(\Delta t, c(t^n)) \; c(t^n) + \nonumber \\
&& + \int_{t^n}^{t^{n+1}} S_{Burgers}(t^{n+1} -s, c(t^{n+1} -s )) \; \sigma(c_{i-1}(s)) \; d W_s ,  
\end{eqnarray}
with initialization $c_0(t) = c(t^n)$ and $i = 1, \ldots, I$.

We deal with the following iterative splitting approaches:
First order iterative splitting approach (related to the AB-splitting approach, means with the rectangle rule and the semi-analytical approach):
\begin{itemize}

\item $i = 0$ (Initialization):
\begin{eqnarray}
c_0(t^{n+1}) = S_{Burgers}(\Delta t, c(t^n)) c(t^n) , 
\end{eqnarray}
where $c_{-1}(t) = 0$.

\item $i = 1$ (first step):
\begin{eqnarray}
&& c_1(t^{n+1}) =  S_{Burgers}(\Delta t, c(t^n)) \; c(t^n) + \nonumber \\
&&  + \int_{t^n}^{t^{n+1}} S_{Burgers}(t^{n+1} - s, c(t^{n+1} -s)) \; \sigma(c_{0}(s)) \; d W_s , 
\end{eqnarray}
where we apply the Ito's rule with a first order scheme (Euler-Maryama-scheme) and obtain:
\begin{eqnarray}
&& c_1(t^{n+1}) =  S_{Burgers}(\Delta t, c(t^n)) c(t^n) + \\
&& +  S_{Burgers}(\Delta t, c(t^n)) \sigma(c_{0}(t^{n+1})) \; \Delta W_{t^{n+1}} , \nonumber 
\end{eqnarray}
where $\Delta W_{t^{n+1}} = W_{t^{n+1}} - W_{t^n} = \sqrt{\Delta t} \; \xi$ and $\xi$ obeys the Gaussian normal distribution $N(0, 1)$ with $\langle \xi \rangle = 0$ and $ \langle \xi^2 \rangle = 1$.

We improve the order to $2$ with the Milstein approach in the stochastic term and obtain:
\begin{eqnarray}
&& c_1(t^{n+1}) =  S_{Burgers}(\Delta t, c(t^n)) c(t^n) + \\
&& +  S_{Burgers}(\Delta t, c(t^n)) \sigma(c_{0}(t^{n+1})) \; \Delta W_{t^{n+1}} + \nonumber \\
&&  + \frac{1}{2} \left( S_{Burgers}(\Delta t, c(t^n)) \sigma(c_{0}(t^{n+1}))  \right) \cdot \nonumber \\
&& \cdot \left(\frac{\partial \left(S_{Burgers}(\Delta t, c(t^n)) \sigma(c)\right)}{\partial c} \right)|_{c_{0}(t^{n+1})}  (\Delta W_{t^{n+1}}^2 - \Delta t) , \nonumber 
\end{eqnarray}
and result to (while $S_{Burgers}(\Delta t, c(t^n))$ is linear and not dependent of $c$, we only have to apply the derivative to $\sigma(c)$).

The algorithm for $i=1$ is given in \ref{algo_1}.
We have to compute the solutions $c(t^{n+1})$ for $n = 0, \ldots, N$.

\begin{algorithm}
\label{algo_1}

We start with the initialization $c(t^0) = c_0$ (initial value)
and $n = 0$.

\begin{enumerate}
\item We compute $i = 1$:
\begin{eqnarray}
&& c_1(t^{n+1}) =  S_{Burgers}(\Delta t, c(t^n)) c(t^n) + \\ 
&& +  S_{Burgers}(\Delta t, c(t^n)) \sigma(c_{0}(t^{n+1})) \; \Delta W_{t^{n+1}} +  \nonumber \\
&&  + \frac{1}{2} \left( S_{Burgers}(\Delta t, c(t^n)) \right)^2 \sigma(c_{0}(t^{n+1})) \cdot \nonumber \\
&& \cdot \left(\frac{\partial \left( \sigma(c)\right)}{\partial c} \right)|_{c_{0}(t^{n+1})}  (\Delta W_{t^{n+1}}^2 - \Delta t) , \nonumber 
\end{eqnarray}
we have $c_{0}(t^{n+1}) = c(t^n)$ as starting value. \\
\item We obtain the next solution $c(t^{n+1}) = c_1(t^{n+1})$, 
If $n = N$, we stop, \\
else we apply $n = n+1$ and goto step 1. \\

\end{enumerate}

\end{algorithm}

\end{itemize}


Second and third order iterative splitting approach (related to the ABA-splitting approach, means with the rectangle rule and the semi-analytical approach):

The next algorithm for $i=2$ is given in \ref{algo_1_1}, we improve the 
last $c_1(t^n)$ with an underlying ABA-method.
We have to compute the solutions $c(t^{n+1})$ for $n = 0, \ldots, N$.

\begin{algorithm}
\label{algo_1_1}

We start with the initialization $c(t^0) = c_0$ (initial value)
and $n = 0$.

\begin{enumerate}
\item We compute $i = 1$:
\begin{eqnarray}
&& \hspace{-2.0cm} c_1(t^{n+1}) =  S_{Burgers}(\Delta t, c(t^n)) c(t^n) +  S_{Burgers}(\Delta t, c(t^n)) \sigma(c(t^{n})) \; \Delta W_{t^{n+1}} +  \\
&&  + \frac{1}{2} \left( S_{Burgers}(\Delta t, c(t^n)) \right)^2 \sigma(c(t^{n})) \left(\frac{\partial \left( \sigma(c)\right)}{\partial c} \right)|_{c(t^{n})}  (\Delta W_{t^{n+1}}^2 - \Delta t) , \nonumber 
\end{eqnarray}
we have $c_{0}(t^{n+1}) = c(t^n)$ as starting value. \\
\item We compute $i = 2$ (with ABA as solution for $c_1(t^n)$):

\begin{eqnarray}
&& \hspace{-1.5cm} c_2(t^{n+1}) =   S_{Burgers}(\Delta t, c_{ABA}(t^n)) c_{ABA}(t^n) + \\
&& +  S_{Burgers}(\Delta t, c_{ABA}(t^n)) \sigma(c_{ABA}(t^{n})) \; \Delta W_{t^{n+1}} + \nonumber \\
&&  + \frac{1}{2} \left( S_{Burgers}(\Delta t, c_{ABA}(t^n)) \right)^2 \sigma(c_{ABA}(t^{n})) \cdot \nonumber \\
&& \cdot \left(\frac{\partial \left( \sigma(c)\right)}{\partial c} \right)|_{c_{ABA}(t^{n})}  (\Delta W_{t^{n+1}}^2 - \Delta t)  , \nonumber 
\end{eqnarray}
we have $\Delta W_{t^{n+1}} = W_{t^{n+1}} - W_{t^n} = \sqrt{\Delta t} \; \xi$ and $\xi$ obeys the Gaussian normal distribution $N(0, 1)$ with $\langle \xi \rangle = 0$ and $ \langle \xi^2 \rangle = 1$. \\
\item We obtain the next solution $c(t^{n+1}) = c_2(t^{n+1})$, 
If $n = N$, we stop, \\
else we apply $n = n+1$ and goto step 1. \\

\end{enumerate}

\end{algorithm}

The next algorithm for $i=2$ is given in \ref{algo_1_2}, we improve the 
last $c_1(t^n)$ with additional intermediate time-steps which are computed by an underlying ABA-method.
We have to compute the solutions $c(t^{n+1})$ for $n = 0, \ldots, N$.

\begin{algorithm}
\label{algo_1_2}

We start with the initialization $c(t^0) = c_0$ (initial value)
and $n = 0$.

\begin{enumerate}
\item We compute $i = 1$:
\begin{eqnarray}
&& \hspace{-2.0cm} c_1(t^{n+1}) =  S_{Burgers}(\Delta t, c(t^n)) c(t^n) +  S_{Burgers}(\Delta t, c(t^n)) \sigma(c(t^{n})) \; \Delta W_{t^{n+1}} +  \\
&&  + \frac{1}{2} \left( S_{Burgers}(\Delta t, c(t^n)) \right)^2 \sigma(c(t^{n})) \left(\frac{\partial \left( \sigma(c)\right)}{\partial c} \right)|_{c(t^{n})}  (\Delta W_{t^{n+1}}^2 - \Delta t) , \nonumber 
\end{eqnarray}
we have $c_{0}(t^{n+1}) = c(t^n)$ as starting value. \\
\item We compute $i = 2$ (with ABA as solution for $c_1(t^n)$):

\begin{eqnarray}
&& \hspace{-1.5cm} c_2(t^{n+1}) =   S_{Burgers}(\Delta t/2, c(t^n)) c(t^n) + \\
&& + S_{Burgers}(\Delta t/2, c_{ABA}(t^{n+1/2})) c_{ABA}(t^{n+1/2}) + \nonumber \\
&& +  S_{Burgers}(\Delta t/2, c(t^n)) \sigma(c(t^{n})) \; \Delta W_{t^{n+1/2}} + \nonumber \\
&& +  S_{Burgers}(\Delta t/2, c_{ABA}(t^{n+1/2})) \sigma(c_{ABA}(t^{n+1/2})) \; \Delta W_{t^{n+1/2}} +  \nonumber \\
&&  + \frac{1}{2} \left( S_{Burgers}(\Delta t/2, c(t^n)) \right)^2 \sigma(c(t^{n})) \left(\frac{\partial \left( \sigma(c)\right)}{\partial c} \right)|_{c(t^{n})}  (\Delta W_{t^{n+1/2}}^2 - \Delta t/2)  + \nonumber \\
&&  + \frac{1}{2} \left( S_{Burgers}(\Delta t/2, c_{ABA}(t^{n+1/2})) \right)^2 \sigma(c_{ABA}(t^{n+1/2})) \cdot \nonumber \\
&& \cdot \left(\frac{\partial \left( \sigma(c)\right)}{\partial c} \right)|_{c_{ABA}(t^{n+1/2})}  (\Delta W_{t^{n+1/2}}^2 - \Delta t/2) , \nonumber 
\end{eqnarray}
we have $c(t^{n+1/2}) = c_{ABA}(t^{n+1/2})$ and  $\Delta W_{t^{n+1/2}} = W_{t^{n+1/2}} - W_{t^n} = \sqrt{\Delta t/2} \; \xi$ and $\xi$ obeys the Gaussian normal distribution $N(0, 1)$ with $\langle \xi \rangle = 0$ and $ \langle \xi^2 \rangle = 1$. \\
\item We obtain the next solution $c(t^{n+1}) = c_2(t^{n+1})$, 
If $n = N$, we stop, \\
else we apply $n = n+1$ and goto step 1. \\

\end{enumerate}

\end{algorithm}

In  figure \ref{fig_algo}, we have the improvement, which are done in Algorithm \ref{algo_1_1} and \ref{algo_1_2}.
\begin{figure}[ht]
\begin{center}  
\includegraphics[width=6.0cm,angle=-0]{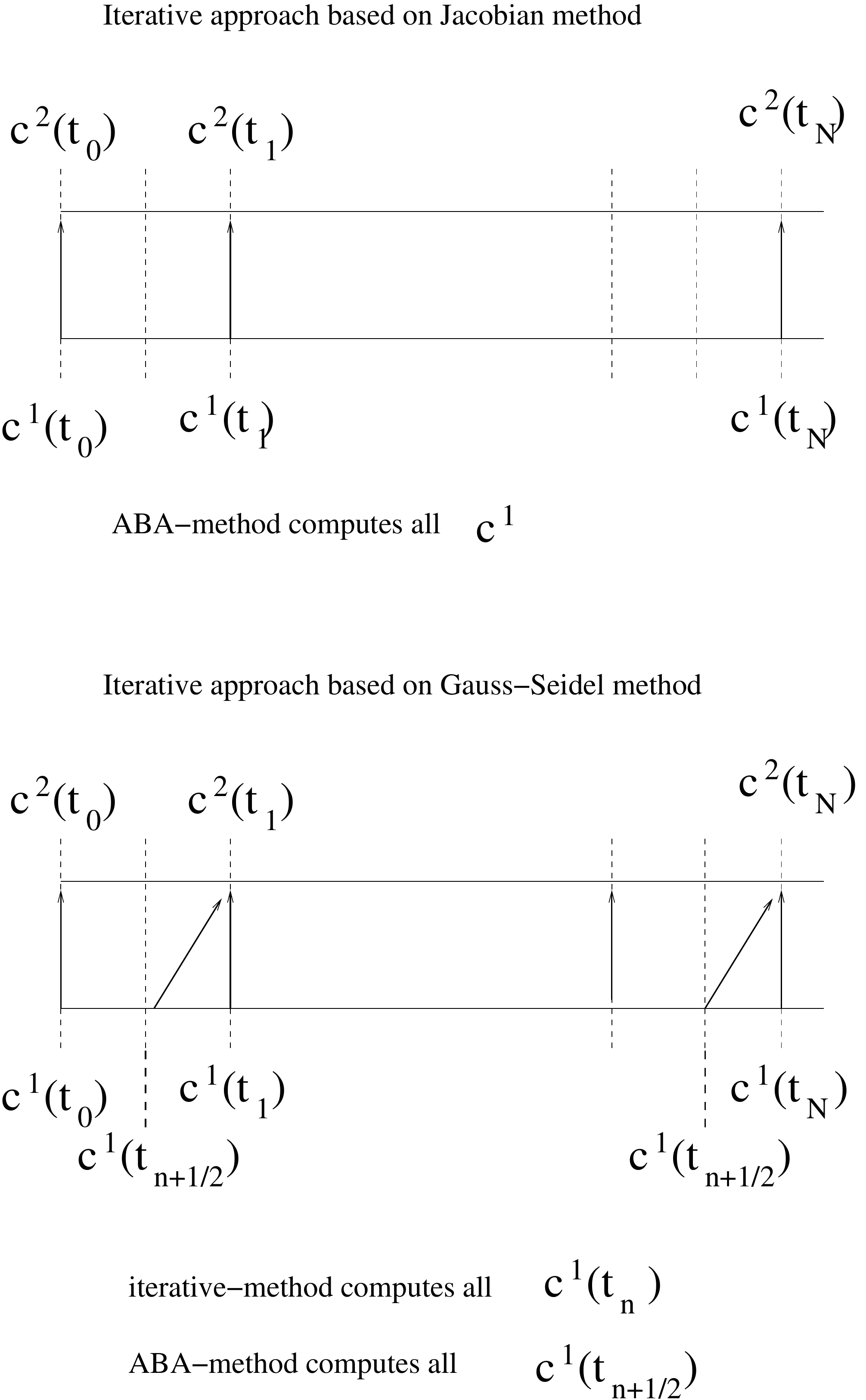}
\end{center}
\caption{\label{fig_algo} Function and visualization of the ABA-splitting approach for the Jacobian- (upper figure) and Gauss-Seidel-method (lower figure).}
\end{figure}

In figure \ref{fig_iter}, we see the further improvements of the iterative approaches.
\begin{figure}[ht]
\begin{center}  
\includegraphics[width=6.0cm,angle=-0]{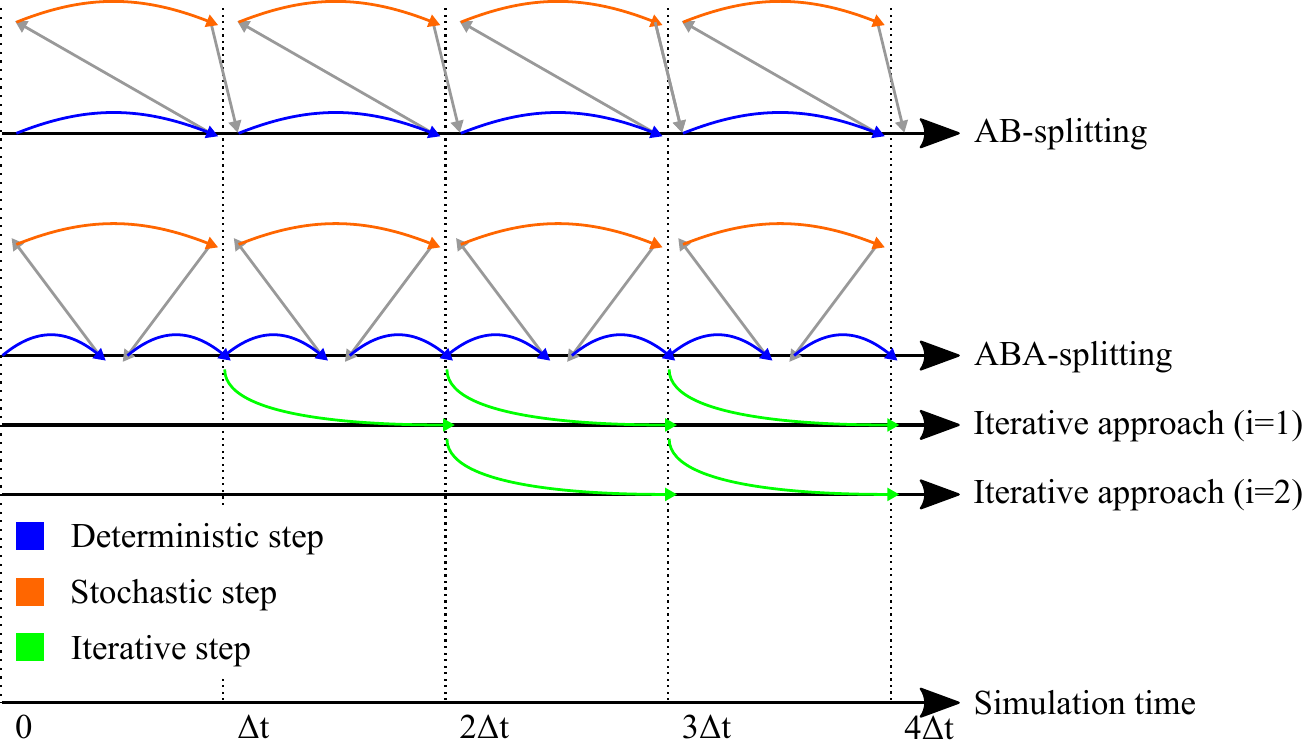}
\end{center}
\caption{\label{fig_iter} Illustration how the iterative algorithms work in principle.}
\end{figure}

\begin{figure}[ht]
\begin{center}  
\includegraphics[width=6.0cm,angle=-0]{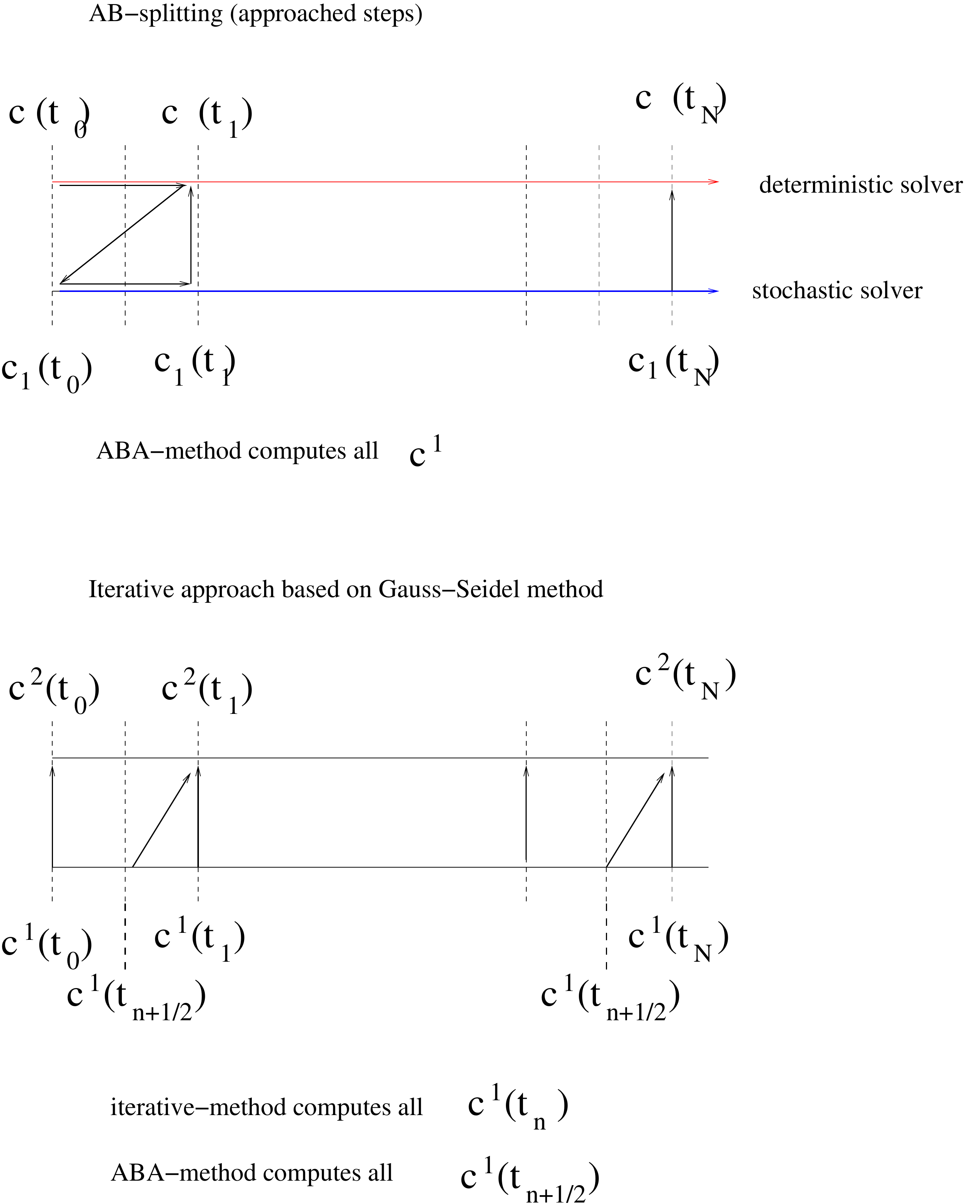}
\end{center}
\caption{\label{fig_iter} The improvements with the iterative approaches with the help of the AB- and ABA splitting approach.}
\end{figure}

\section{Numerical Analysis}
\label{analysis}

In the numerical analysis, we concentrate on the new iterative algorithms 
and present the approximation to the fixpoint of the solutions.

The iterative splitting scheme is given as:
\begin{eqnarray}
&& \frac{\partial c_i}{\partial t} dt =  - \frac{\partial f(c_{i-1}, c_i)}{\partial x} + \sigma(c_{i-1}(t)) \; d W_t , \; c_i(x, t^n) = c(x, t^n) ,
\end{eqnarray}
where we have $i = 1, 2, \ldots, I$ with the start condition $c_0(t) = c(t^n)$.

We apply the integration and have the solution
\begin{eqnarray}
\label{iter_fix}
&& \int_{t^n}^{t^{n+1}} \frac{\partial c_i}{\partial t} dt =  - \int_{t^n}^{t^{n+1}}  \frac{\partial f(c_{i-1}, c_i)}{\partial x} +   \int_{t^n}^{t^{n+1}} \sigma(c_{i-1}(t)) \; d W_t , \; c_i(x, t^n) = c(x, t^n) ,  \nonumber \\
&& c_i(t^{n+1}) = S_{Burgers}(\Delta t, c_{i-1}(t^n)) \; c(t^n) + \nonumber \\
&& + \int_{t^n}^{t^{n+1}} S_{Burgers}(t^{n+1} -s, c_{i-1}(t^{n+1} -s )) \; \sigma(c_{i-1}(s)) \; d W_s ,  
\end{eqnarray}
with initialization $c_0(t) = c(t^n)$ and $i = 1, \ldots, I$.

\begin{definition}
We have $\Omega \subset \R^N$ and $G : \Omega \rightarrow \R^M$.
Further, $G$ is Lipschitz contiuous on $\Omega$ with Lipschitz-constant $\gamma$
if
\begin{eqnarray}
|| G(x) - G(y) || \le \gamma || x - y||,
\end{eqnarray}
for all $x, y \in \Omega$.
\end{definition}

We have the following assuptions:
\begin{assumption}
We have the Lipschitz-constinuous functions $S_b(c)$ and $\sigma(c)$,
while we also assume $\frac{\partial \sigma}{\partial c}$ is Lipschitz continuous.
\end{assumption}

Then, we have the following Lemma:
\begin{lemma}
We have $\Omega \subset \R^N$ and $S_b, \frac{\partial \sigma}{\partial c} : \Omega \rightarrow \R^M$.
Further, $S_b$ and  $\frac{\partial \sigma}{\partial c}$ are  contraction mappings on $\Omega$, while we assume $\frac{\partial \sigma}{\partial c}$ are Lipschitz continuous with constants $\gamma_1 \le 1$ and $\gamma_2 \le 1$.
\end{lemma}

\begin{proof}
We have 
\begin{eqnarray}
|| S_b(x^n) x - S_b(y^n) y || \le \gamma_1 || x - y||,
\end{eqnarray}
While the operator $S_b$ for the pure deterministic Burgers' equation is
bounded with respect to $\Delta x$ and $\Delta t$ we obtain $|| S_b(x^n)|| \le \gamma_1$ for sufficient small $\Delta x$ and $\Delta t$.

Further, we have
\begin{eqnarray}
||  \frac{\partial \sigma}{\partial x}|_{x = x^n} x -  \frac{\partial \sigma}{\partial y}|_{y = y^n} y || \le \gamma_2 || x - y||,
\end{eqnarray}
while the operator $\frac{\partial \sigma}{\partial x}$ is bounded and
lipschitz continuous.

\end{proof}

\begin{theorem}

$\Omega$ is a closed subset of $\R^N$ and $S_B$ and $\frac{\partial \sigma}{\partial x}$
are contraction mappings on $\Omega$ with Lipschitz-constants 
$\gamma_1 < 1$ and $\gamma_2 < 1$, then the iterative scheme (\ref{iter_fix})
converge linearly to $x^*$ with the factor $\tilde{\gamma}$.
\end{theorem}

\begin{proof}

We apply the iterative scheme:
\begin{eqnarray}
&& \int_{t^n}^{t^{n+1}} \frac{\partial c_i}{\partial t} dt =  - \int_{t^n}^{t^{n+1}}  \frac{\partial f(c_{i-1}, c_i)}{\partial x} +   \int_{t^n}^{t^{n+1}} \sigma(c_{i-1}(t)) \; d W_t , \; c_i(x, t^n) = c(x, t^n) ,  \nonumber \\
&& c_i(t^{n+1}) = S_{Burgers}(\Delta t, c_{i-1}(t^n)) \; c(t^n) + \nonumber \\
&& + \int_{t^n}^{t^{n+1}} S_{Burgers}(t^{n+1} -s, c_{i-1}(t^{n+1} -s )) \; \sigma(c_{i-1}(s)) \; d W_s ,  
\end{eqnarray}
and evalutate the integral based on the Taylor-Ito scheme at the integration
point $t^n$ and obtain, accuracy of the Milstein-scheme, which is given as:
\begin{eqnarray}
&& c_i(t^{n+1}) =  S_{Burgers}(\Delta t, c_{i-1}(t^n)) c(t^n) + \\
&& +  S_{Burgers}(\Delta t, c_{i-1}(t^n)) \sigma(c(t^{n})) \; \Delta W_{t^{n+1}} + \nonumber  \\
&&  + \frac{1}{2} \left( S_{Burgers}(\Delta t, c_{i-1}(t^n)) \right)^2 \sigma(c(t^{n})) \left(\frac{\partial \left( \sigma(c)\right)}{\partial c} \right)|_{c_{i-1}(t^{n})}  (\Delta W_{t^{n+1}}^2 - \Delta t) . \nonumber
\end{eqnarray}

We apply:
\begin{eqnarray}
&& || c_i(t^{n+1}) - c_{i-1}(t^{n+1})|| \le \nonumber \\ 
&& \le ||  S_{Burgers}(\Delta t, c_{i-1}(t^n)) - S_{Burgers}(\Delta t, c_{i-2}(t^n)) || || c(t^n) || + \nonumber \\
&& || S_{Burgers}(\Delta t, c_{i-1}(t^n))- S_{Burgers}(\Delta t, c_{i-2}(t^n)) || || \sigma(c(t^{n})) \; \Delta W_{t^{n+1}} || + \nonumber \\
&& + \frac{1}{2} || \left( S_{Burgers}(\Delta t, c_{i-1}(t^n)) \right)^2  \left( \frac{\partial \left( \sigma(c)\right)}{\partial c} \right)|_{c_{i-1}(t^{n})} - \nonumber \\
&& -  \left( S_{Burgers}(\Delta t, c_{i-2}(t^n)) \right)^2 \left( \frac{\partial \left( \sigma(c)\right)}{\partial c} \right)|_{c_{i-2}(t^{n})} ||  \cdot \nonumber \\
&& \cdot || \sigma(c(t^{n})) \;  (\Delta W_{t^{n+1}}^2 - \Delta t) || \le \nonumber \\
&& \le \gamma_1 ||c_{i-1}(t^{n}) - c_{i-2}(t^{n}) || C_1 +  \gamma_1 ||c_{i-1}(t^{n}) - c_{i-2}(t^{n}) || C_2 + \nonumber \\
&& + \gamma_1 \gamma_2 ||c_{i-1}(t^{n}) - c_{i-2}(t^{n}) || C_3  \le \nonumber \\
&& \le \tilde{\gamma} ||c_{i-1}(t^{n}) - c_{i-2}(t^{n}) || \tilde{C} ,
\end{eqnarray}
where we assume the constants $C_1, C_2, C_3$ and $\tilde{C}$ are bounded
and $\tilde{\gamma} < 1$ with sufficient small $\Delta t$ and $\Delta x$.

We apply the recursion and obtain:
\begin{eqnarray}
&& || c_i(t^{n+1}) - c_{i-1}(t^{n+1})|| \le \tilde{\gamma}^i ||c_{i-1}(t^{n+1-i}) - c_{i-2}(t^{n+1-i}) || \tilde{C}^i ,
\end{eqnarray}
where we obtain $\lim_{n, i \rightarrow \infty} || c_i(t^{n+1}) - c_{i-1}(t^{n+1})|| \rightarrow 0$.

\end{proof}

\begin{remark}
We obtain a convergence to the fixpoint of the equation based on the iterative scheme. We also obtain an acceleration of the solver-process and a reduction of the numerical error with additional iterative steps.  
\end{remark}

\subsection{Numerical Error Analysis (strong and weak errors)}

For the verification of the theoretical results for the
iterative splitting scheme in section \ref{num}, we deal with the following
numerical error analysis.

We present the convergence rates of the following weak
errors:
\begin{itemize}
\item Weak errors:
\begin{eqnarray}
&&  err_{weak, \Delta t} = | E(c_{\Delta t, Scheme}) - c_{\Delta t, pure burg}) | , \nonumber \\
 && = \left|  \left( \frac{1}{N} \sum_{j=1}^N c_{\Delta t, Scheme, j} \right) - c_{\Delta t, pure burg } \right|,
\end{eqnarray}
where $N$ are the number of solutions of the stochastic Burgers' equation.

Further $c_{\Delta t, pure burg }$ is the solution of the pure Burgers' equation and
$c_{\Delta t, Scheme, j}$ is the $j$-th solution of the stochastic Burgers' equation.
We assume to have $10 -100$ runs of the stochastic Burgers' equation.

\item Variance for the solution at $t = t_n$ and $N_s$-sample paths:
\begin{eqnarray}
\label{estimated_2}
&& Var(c_{\Delta t, Scheme}(t_n)) = E( \left(c_{\Delta t, Scheme}(t_n) - E(c_{\Delta t, Scheme}(t_n)) \right)^2) = \nonumber \\
&&  = \left( \frac{1}{N} \sum_{j=1}^{N} (c_{\Delta t, Scheme}(t_n))^2 \right) -  E(c_{\Delta t, Scheme, j}(t_n)) ,
\end{eqnarray}
we deal with $N_s$ number of seeds and  $method = \{AB, ABA, BAB, iter \}$, $c_{\Delta t, Scheme,j}$ is the result of the
method at $t^n$ in the seed $j$. Further, we apply for the iterative scheme $iter=1, \ldots, iter=4$ steps.

\end{itemize}

\begin{remark}
In the numerical examples, we obtain the weak convergence rates with the weak error.
Here, we also apply the weak error to obtain an overview to the
accuracy of the numerical schemes.
\end{remark}

\section{Numerical experiments}
\label{numerics}

In the following numerical experiments, we concentrate on pure stochastic Burgers' equation, which is given as:
\begin{eqnarray}
&& \frac{\partial c}{\partial t} +  \frac{\partial f(c)}{\partial x} = \sigma(c) \frac{\partial W}{\partial t} ,  \; (x, t) \in [0, X] \times [0, T] , \\
&& c(x, 0) = c_0(x) , \; x \in [0, X] .
\end{eqnarray}

For the discretization with Finite Difference or Finite Volumes,
we deal with the CFL condition of the two explicit discretized terms as:
\begin{itemize}
\item
$\Delta t \le \frac{\Delta x}{|u_i^{n-1}|}$, for the deterministic part, 
\item
$\Delta t \le \frac{(u_i^{n-1})^2}{(\sigma(u_i^{n-1}))^2 \xi^2} $, for the stochastic part, 
\item $\Delta t \le \left( \frac{1}{\frac{|u_i^{n-1}|}{\Delta x} +  \frac{\sigma(u_i^{n-1}) |\xi|}{u_i^{n-1}} } \right)^2$, for both terms and we assume $\sqrt{\Delta t} \le \Delta t$ ,
\end{itemize}
where we use estimates for $u_i^{n-1}$ and $\Delta W = \sqrt{\Delta t} \; \xi$ with $\xi$ is Gaussian normal distributed.

We apply the following methods, that we discussed in section \ref{num}:
\begin{itemize}
\item AB-splitting,
\item ABA-splitting,
\item BAB-splitting,
\item iter-splitting (where we apply $i=1, \ldots, 4$ iterative steps).
\end{itemize}

The $L_1$-errors of the different numerical schemes are given in figure \ref{err_AB}.
\begin{figure}[ht]
\begin{center} 
\includegraphics[width=10.0cm,angle=-0]{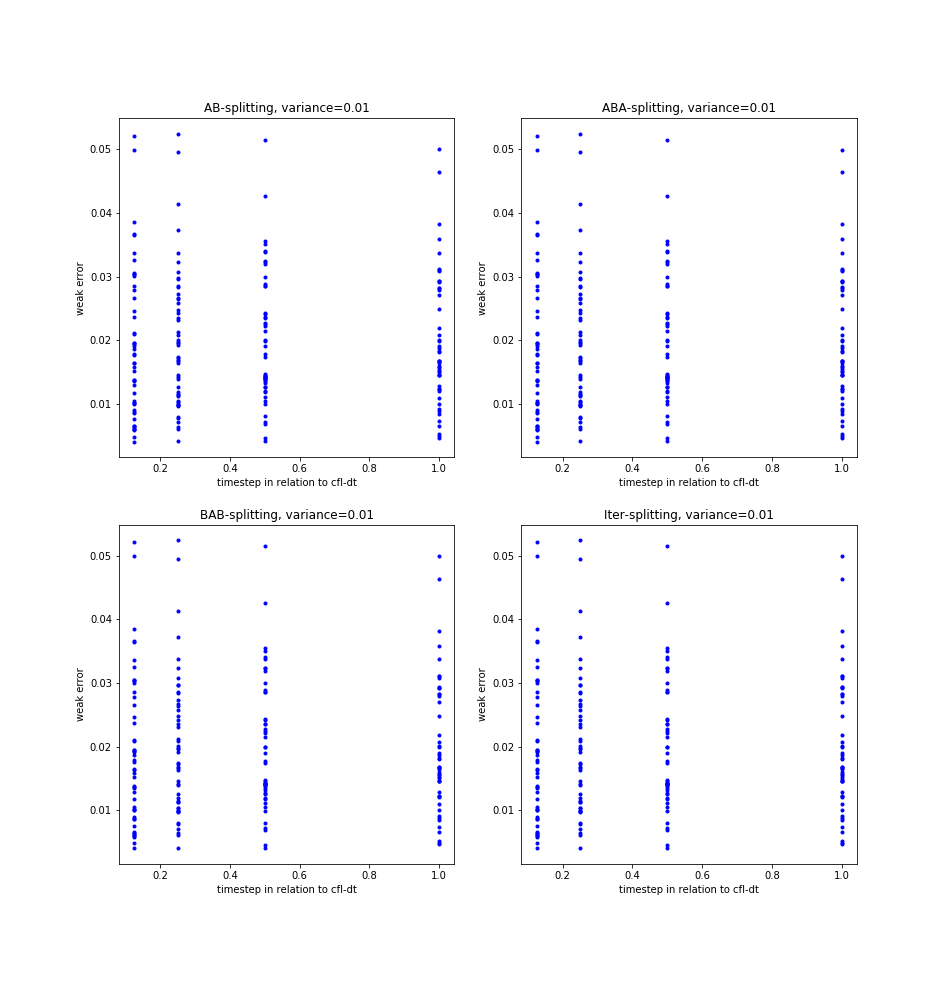}
\end{center}
\caption{\label{err_AB} Left upper figure: $L1$-errors of the AB-spitting approach, right upper figure: $L1$-errors of the ABA-spitting approach, left lower figure: $L1$-errors of the BAB-spitting approach and right lower figure: $L1$-errors of the itertative-spitting approach.}
\end{figure}

The comparison of the schemes with the iterative splitting approach is given in figure \ref{split_com_iter}.
\begin{figure}[ht]
\begin{center}  
\includegraphics[width=10.0cm,angle=-0]{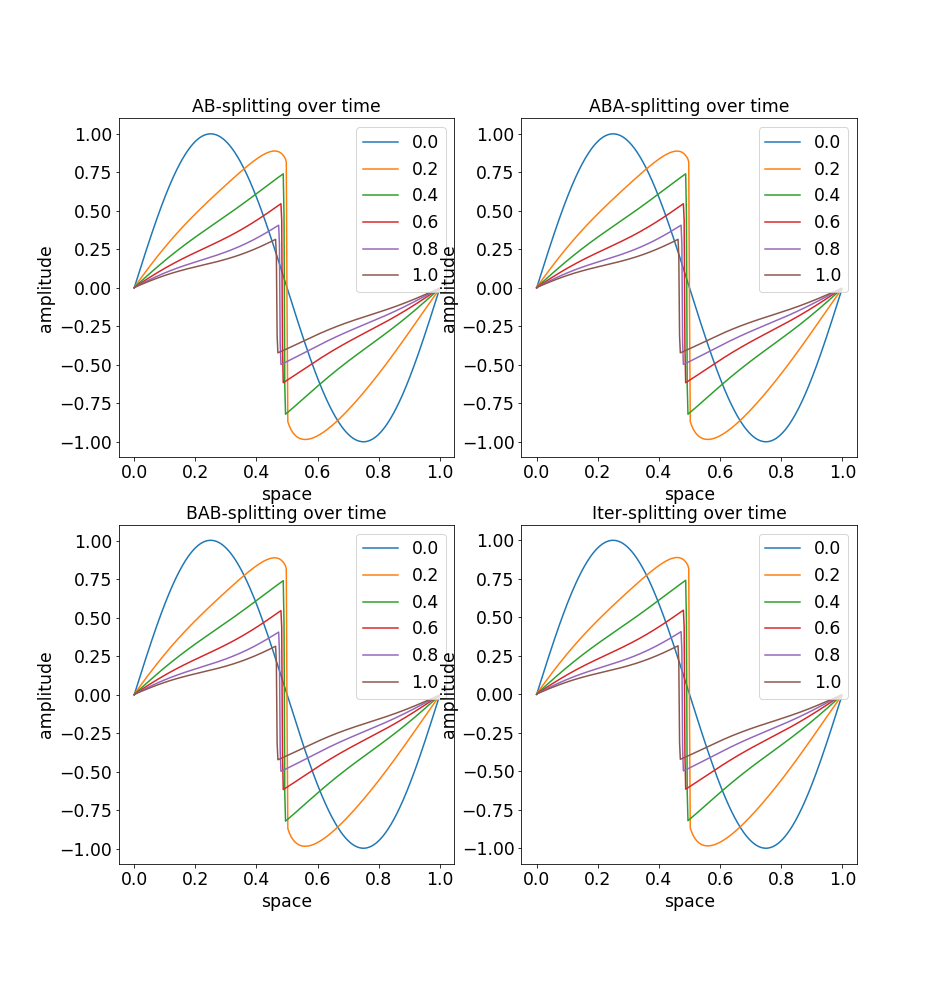}
\end{center}
\caption{\label{split_com_iter} Simulation-results of the different methods and the improved iterated approaches.}
\end{figure}

The $L1$ errors are defined as errors to compare to the weak and strong
convergence, while $ err_{weak}$ is defined to see the error of the
averaged stochastic solution to the deterministic solution.
The errors of the schemes are presented in figure \ref{iter_error}.
\begin{figure}[ht]
\begin{center}  
\includegraphics[width=10.0cm,angle=-0]{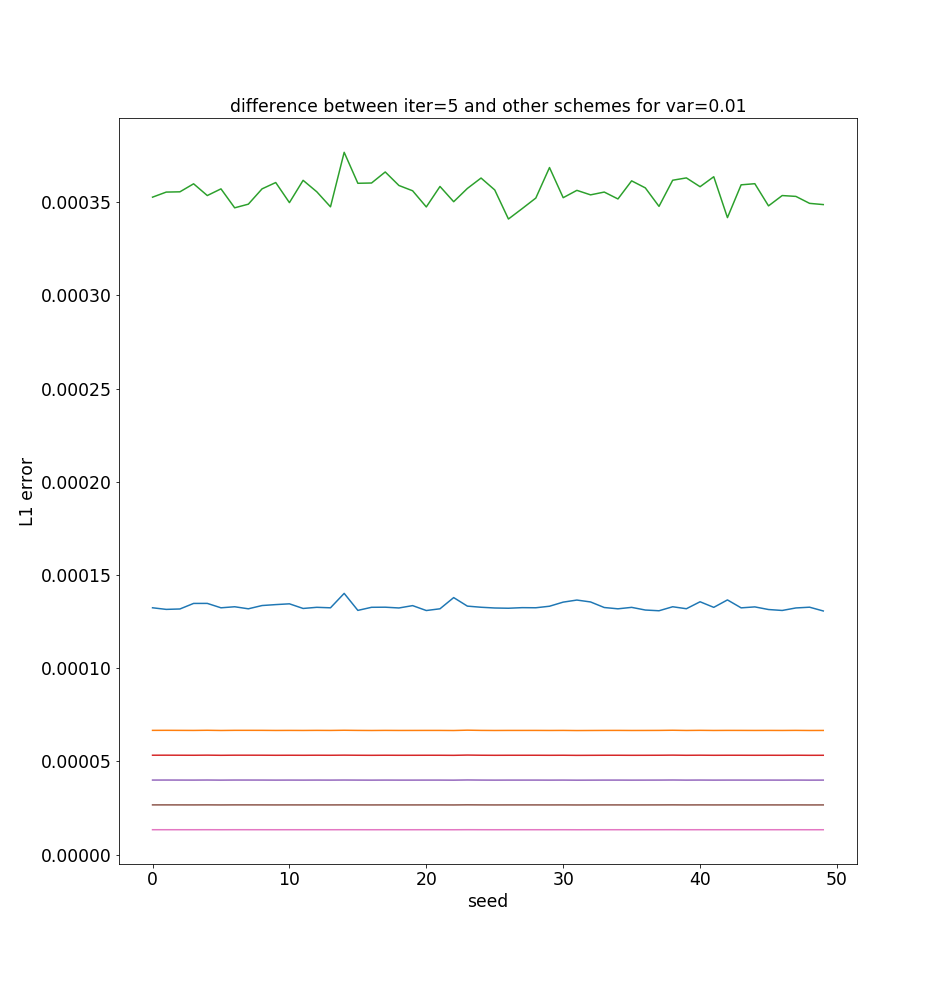}
\end{center}
\caption{\label{iter_error} $L_1$-error of the simulation-results of the different methods with different iterative steps.}
\end{figure}

The seeds and the blow-ups of the noniterative splitting approaches
are presented in \ref{iter_error_2}.
For the experiments and the numerical tests, we initialize the pseudo random generator for each new perturbation with a new seed.

\begin{figure}[ht]
\begin{center}  
\includegraphics[width=8.0cm,angle=-0]{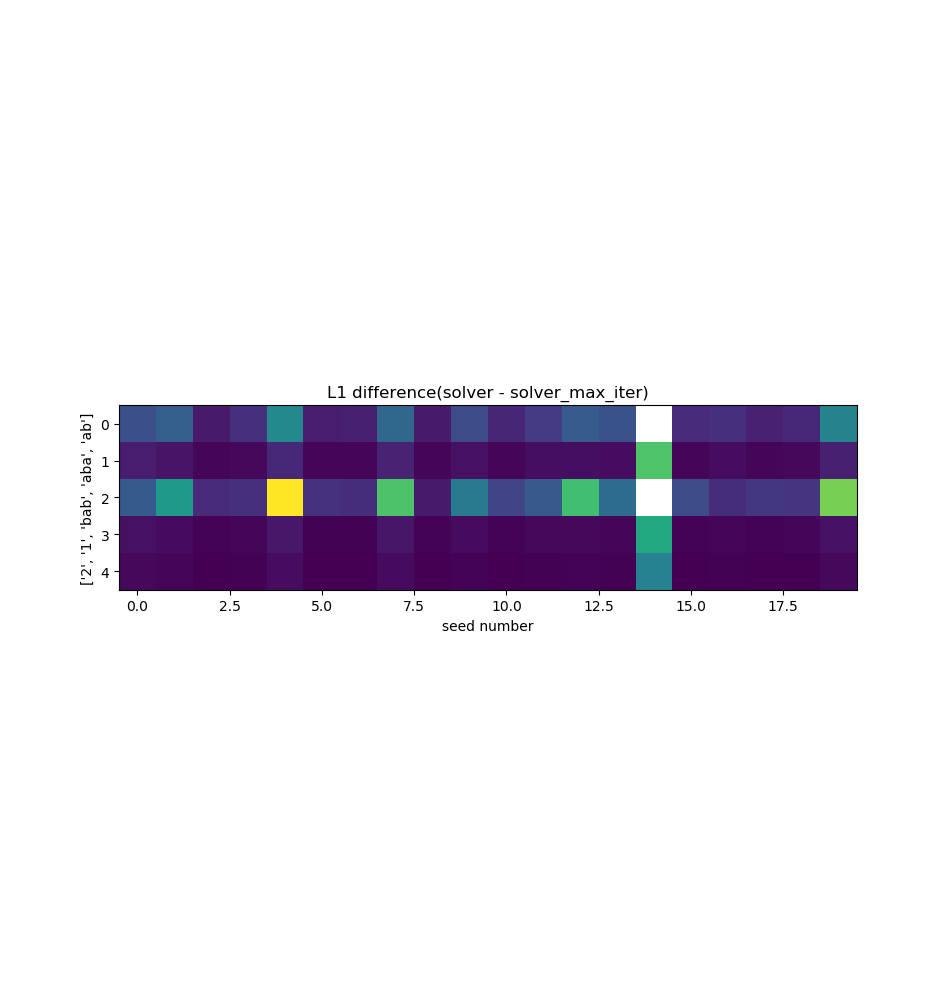}
\includegraphics[width=8.0cm,angle=-0]{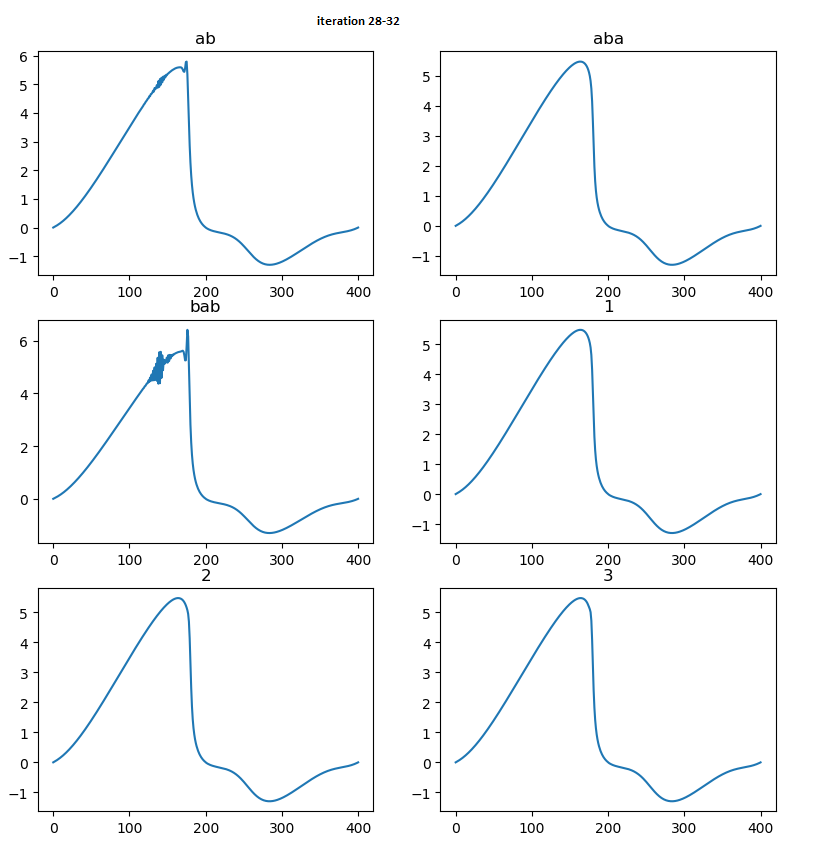}
\end{center}
\caption{\label{iter_error_2} In the left figure, we see the different seeds of the non-iterative and iterative approaches. The right figure presents the benefits of the iterative approaches, while in the non-iterative approaches, we have blow-up.}
\end{figure}

The numerical errors between the perturbed and unperturbed solution and the averaged results are given in figure \ref{iter_error_3}.
\begin{figure}[ht]
\begin{center}  
\includegraphics[width=8.0cm,angle=-0]{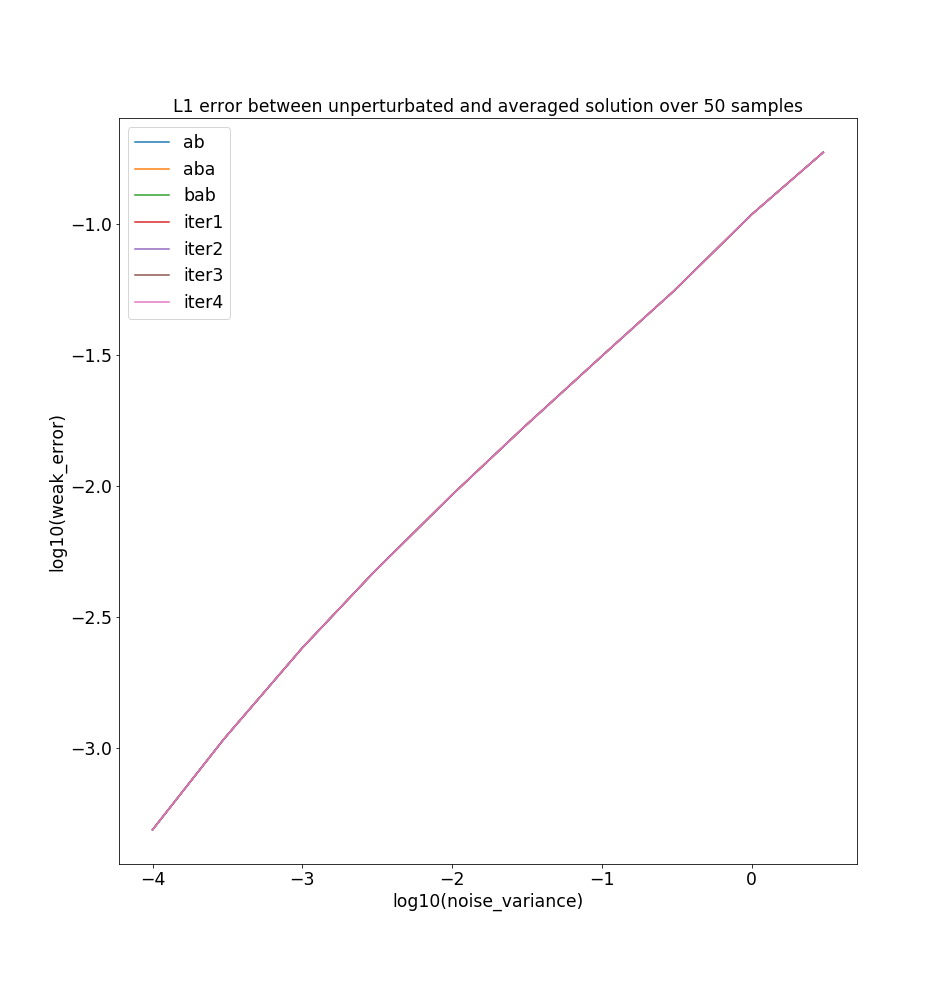}
\includegraphics[width=8.0cm,angle=-0]{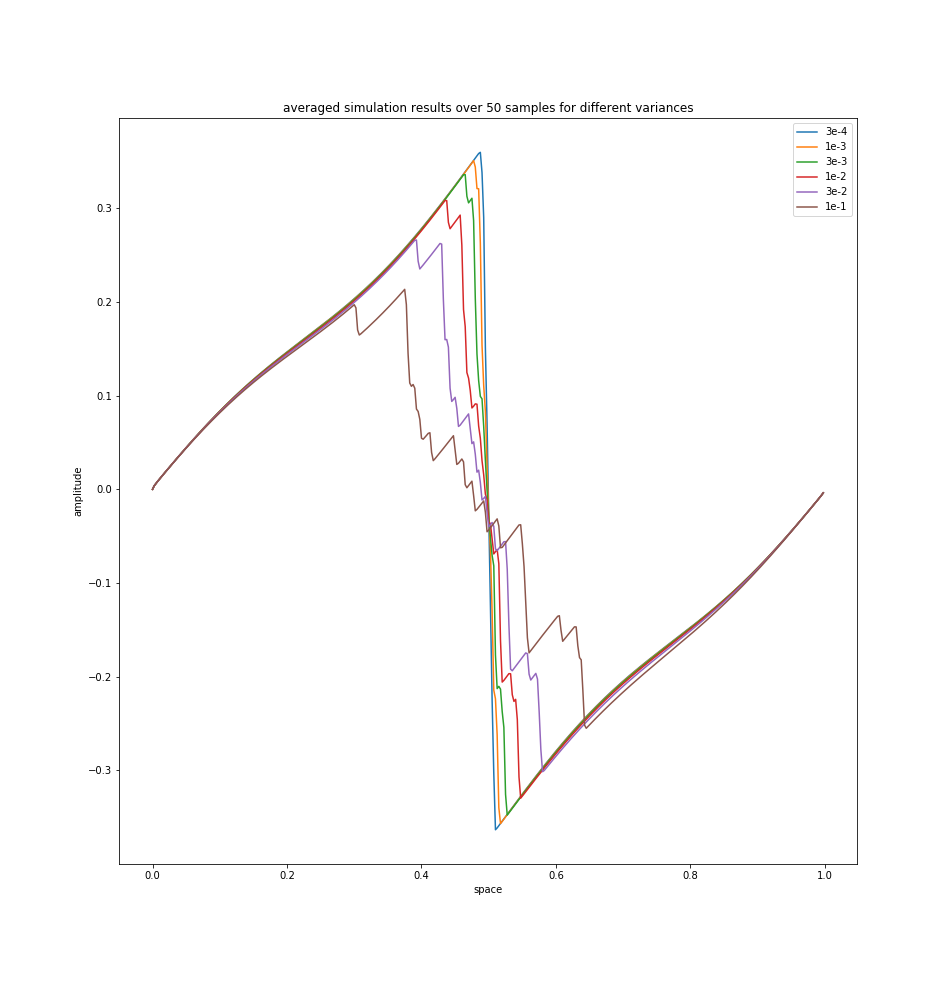}
\end{center}
\caption{\label{iter_error_3} In the left figure, we see $L_1$ error between unperturbated solution and average over $50$ perturbed solutions, only AB-splitting, dependent on variances. The right figure presents the averaged results for difference variances.}
\end{figure}

\begin{remark}
With the iterative splitting schemes, we obtain more accurate results and reduce the numerical errors in each iterative step.
\end{remark}

\section{Conclusion}
\label{concl}

We presented new iterative methods based on the Picard's approximation. Such methods allow to 
obtain more accurate results, while we could reduce the error with the iterative steps.
We presented first numerical results with the stochastic Burgers' equation and multiplicative noise.
In future, we also apply more delicate SPDEs with respect to mixed time and spatial noises.

\section{Appendix}

In the following, we add some more details to the
iterative and non-iterative schemes.

\subsection{Derivations of the methods}

The SDE is given as:
\begin{equation}
\begin{aligned}
u(x,t)_t + 0.5 (u(x,t)^2)_x &= \sigma(u(x,t)) W(t)_t\\
u_t + 0.5 (u^2)_x &= \sigma(u) W_t,
\end{aligned}
\end{equation}
$u = u(x,t)$ is a spatial dependent amplitude, that changes over time. $W=W(t)$ is a Wiener process (Standard Brownian motion, summation of white noise) over time and $W_t$ is its time derivative; so basically white noise. $\sigma(u,t)$ is the diffusion coefficient and can further be time variant or dependent on past values of $u$. For the sake of brevity, some variables stay omitted.

Ignoring the stochastic part on the right, one can transform the deterministic part and match it against the standard Burgers' equation (missing diffusion term $v=0$ in this case),

\begin{equation}
\begin{aligned}
u_t + 0.5 \frac{\partial}{\partial x}(u^2) = u_t + u u_x &= 0 \\
u_t + u u_x &= v \frac{\partial^2 u}{\partial x^2}
\end{aligned}
\end{equation}

which is solved on a discrete time grid by the conservation law solver with Engquist-Osher in the function $SCL$ (solve conservation law).
\\

A stochastic equation of the form 
\begin{equation}
d X_t = \mu X dt + \sigma X dW_t
\end{equation}

with real constants $\mu$ and $\sigma$ can be solved by the Milstein method and has the numerical solution 
\begin{equation}
X_{t+\Delta t} = X_t + a(X_t) \Delta t + b(X_t) \Delta W_t + \frac{1}{2} b(X_t) b^\prime(X_t)((\Delta W_t)^2 - \Delta t)
\end{equation}
 with drift $a(x)= \mu x$ and diffusion $b(x)= \sigma x$. Via pattern matching, one can deduce a numerical solution for the stochastic part:

\begin{equation}
\begin{aligned}
u_{t+\Delta t}  &= u_t + \sigma(u_t)W_t \\
u_{t+\Delta t} &= u_t + \sigma(u_t) \Delta W_t + \frac{1}{2} \sigma(u_t) \sigma^\prime(u_t)((\Delta W_t)^2 - \Delta t) .
\end{aligned}
\end{equation}

Note that the notation of an index like $_t$ may refer to a derivative in continuous time or time stamp in discrete time. This solver is implemented in the function $SSDE$ (solve stochastic differential equation).

\subsection{AB-splitting}

The AB-splitting approach divides the time scale into $N$ intervals. Each interval is further split into $J$ subintervals. The AB-splitting takes an initial condition $u(t)$ and solves the deterministic problem on the subinterval $J$ to obtain a solution $u(t+\Delta t)_1$. The deterministic result is used as initial condition for the stochastic solver, which calculates the final result $u(t +\Delta t)$. This process is repeated $N$ times to obtain the final result.

\subsection{ABA-splitting}

This algorithm works almost identically to the upper one, but with a slight modification with regards to the order and length of the solvers. The deterministic solver is used on the first half of the subinterval, resulting in a helper solution $u(t + \frac{1}{2} \Delta t)_1$. The stochastic solver than proceeds to calculate a solution on the whole subinterval and yields another solution $u(t+ \Delta t)_2$, which is used as initial condition for the deterministic solver in order to yield the final result $u(t+\Delta t)$ for the second half of the subinterval.

\subsection{BAB-splitting}

The same as above, but stochastic and deterministic solver are exchanged.

\subsection{iterative splitting (after discretization)}

The iterative scheme, here we apply the iterative steps of
a stochastic equation of the form 
\begin{equation}
d X_{i,t} = a(X) dt + b(X) dW_t
\end{equation}
while we obtain an analytical solution of $d X_{i,t} = a(X_i) dt $, which is $X_{t+\Delta t} = S(a(X_t), \Delta t) X_{t}$.

We apply the Milstein-scheme plus a fixpoint iterative scheme,
which is given as
\begin{eqnarray}
&& X_i(t^{n+1}) =  S_{Burgers}(\Delta t, a(X_{i-1}^{n+1})) X^n + b(X_{i-1}^{n+1}) \Delta W + \nonumber \\ 
&& +  b(X_{i-1}^{n+1}) \frac{\partial b(X)}{\partial X}|_{X = X_{i-1}^{n+1}} (\Delta W^2 - \Delta t) ,
\end{eqnarray}
 with drift $a(x)$ and diffusion $b(x)$.

\subsection{iterative splitting (before discretization)}

 iterative Steps $i > 1$):
\begin{eqnarray}
&& \hspace{-1.5cm} c_i(t^{n+1}) =  S_{Burgers}(\Delta t, c(t^n)) c(t^n) +  S_{Burgers}(\Delta t, c(t^n)) \sigma(c_{i-1}(t^{n+1})) \; \Delta W_{t^{n+1}} +  \\
&&  + \frac{1}{2} \left( S_{Burgers}(\Delta t, c(t^n)) \right)^2 \sigma(c_{i-1}(t^{n+1})) \left(\frac{\partial \left( \sigma(c)\right)}{\partial c} \right)|_{c_{i-1}(t^{n+1})}  (\Delta W_{t^{n+1}}^2 - \Delta t) , \nonumber 
\end{eqnarray}

\subsection{Improvements be the integration of the variation of constants}

$i = 2$ (trapezoidal rule):
\begin{eqnarray}
&& c_2(t^{n+1}) = S_{Burgers}(\Delta t) c(t^n) - \nonumber \\ 
&& - \frac{1}{2} \sigma(c_1(t^{n+1}))^2 \; \Delta t + \sigma(\frac{c_1(t^{n+1}) + c(t^{n})}{2}) \; \Delta W_t .  
\end{eqnarray}

$i > 2$:
\begin{eqnarray}
&& c_i(t^{n+1}) = S_{Burgers}(\Delta t) c(t^n) - \nonumber \\ 
&&  - \frac{1}{2} \sigma(c_1(t^{n+1}))^2 \; \Delta t + \sigma(\frac{c_{i-1}(t^{n+1}) + c(t^{n})}{2}) \; \Delta W_t .  
\end{eqnarray}
with $i = 3, \ldots, I$.

\subsection{Numerical Errors: Weak and Strong error}

\begin{itemize}
\item Strong error:
\begin{eqnarray}
&& {L_1}_{strong, \Delta t, \Delta x, s} = E(|U_{ref} - U_{scheme}|) = \nonumber \\
&& =  \sum_{s = 1}^{N_s} ( \sum_{j = 1}^N | u_{\Delta t, \Delta x, j, reference, s} - u_{\Delta t, \Delta x, j, scheme, s} |  /N ) / N_s,
\end{eqnarray}
where $s$ is the index of the different seeds, meaning $W(s)$ (different Wiener-processes) and $N=N_t \cdot N_x$ is the number of the $N_t$ temporal and $N_x$ spatial steps. Furthermore $reference$ is the deterministic reference solution and we test the schemes $ scheme = \{AB, \; ABA, \; BAB, \; iter\}$.
\item Weak error:
\begin{eqnarray}
&& {L_1}_{weak, \Delta t, \Delta x, s} = | E(f(U_{ref}) - E(f(U_{scheme}))| = \\
&& = | \sum_{s = 1}^{N_s} ( \sum_{j = 1}^N  u_{\Delta t, \Delta x, j, reference,s} / N ) / N_s  -  \sum_{s = 1}^{N_s} ( \sum_{j = 1}^N  u_{\Delta t, \Delta x, i, scheme,s} / N ) / N_s | , \nonumber
\end{eqnarray}
where we assume $f(u) = u$, further $s$ is the index of the different seeds for $W(s)$ (different Wiener-processes) and $N_s$ are the number of seeds.
$N=N_t \cdot N_x$ is the number of the $N_t$ temporal and $N_x$ spatial steps. 
Additionally $reference$ is the deterministic reference solution and we test the schemes $ scheme = \{AB, \; ABA, \; BAB, \; iter\}$.
The weak error is defined with the average value of the stochastic result, while the reference solution is the deterministic result.
\end{itemize}

\section*{Acknowledgments}\label{sec:Acknowledgments}

The authors would like to thank Dr. Erlend Briseid Storrosten (University of Oslo, Norwey) for his python-code. Based on his code and the discussion with him, we could modify the experiments for our splitting approaches.

\bibliographystyle{plain}

\end{document}